\documentclass[12pt,reqno]{amsart}

\usepackage{amssymb}

\usepackage[all]{xy}

\numberwithin{equation}{section}

\newtheorem{theorem}{Theorem}[section]
\newtheorem{proposition}[theorem]{Proposition}
\newtheorem{lemma}[theorem]{Lemma}

\theoremstyle{definition}
\newtheorem{definition}[theorem]{Definition}
\newtheorem{remark}[theorem]{Remark}

\begin{document}

\baselineskip=15pt

\title[Symplectic structure and Sasakian threefolds]{Symplectic structure on the character
varieties of Sasakian threefolds}

\author[I. Biswas]{Indranil Biswas}

\address{Department of Mathematics, Shiv Nadar University, NH91, Tehsil Dadri,
Greater Noida, Uttar Pradesh 201314, India}

\email{indranil.biswas@snu.edu.in, indranil29@gmail.com}

\author[A. N. Sengupta]{Ambar N. Sengupta}

\address{Department of Mathematics, University of Connecticut, Storrs, CT 06269, USA}

\email{ambarnsg@gmail.com}



\date{}

\begin{abstract}
Take a compact Sasakian threefold $M$ and consider the associated irreducible $\text{SL}(r,{\mathbb C})$--character
variety ${\mathcal R}\, :=\, \text{Hom}(\pi_1(M,\, x_0), \, \text{SL}(r, {\mathbb C}))^{ir}/
\text{SL}(r, {\mathbb C})$ of $M$, where $\text{Hom}(\pi_1(M,\, x_0), \, \text{SL}(r, {\mathbb C}))^{ir}$ is the
space of irreducible homomorphisms. We first construct a natural algebraic $2$--form on $\mathcal R$. Then
it is shown that this $2$--form is closed. Finally we show that the restriction of this
$2$--form to $\text{Hom}(\pi_1(M,\, x_0), \, \text{SU}(r))^{ir}$ is symplectic.
\end{abstract}

\maketitle

\section{Introduction}

Sasakian geometry is generally qualified as an odd-dimensional ``analog" of K\"ahler
geometry. In fact, just as the contact manifolds serve as substitutes for symplectic manifolds in odd
dimensions, the Sasakian manifolds are the odd dimensional counterparts of K\"ahler manifolds.
Although they are around for quite some time, in recent years they were extensively
studied. This resurgence is substantially influenced by the recently found relevance of
Sasakian manifolds in string theory. C. Boyer and K. Galicki published a series of papers
investigating various differential geometric aspects of Sasakian manifolds (see \cite{BG} and
references therein). Sasakian manifolds appeared in string theory through the work of J.
Maldacena \cite{Ma}; see the papers of J. Sparks, \cite{Sp1}, \cite{Sp2}, \cite{Sp3}, \cite{MS},
\cite{GMSY}, \cite{MSY}, and references therein for more details.

We consider compact Sasakian manifolds of dimension three. Take a three dimensional
compact Sasakian manifold $M$, and consider the irreducible homomorphisms from 
$\pi_1(M,\, x_0)$, where $x_0\, \in\, M$ is a base point, to $\text{SL}(r, {\mathbb C})$. Each
connected component of the irreducible character variety ${\mathcal R}\, :=\,
\text{Hom}(\pi_1(M,\, x_0), \, \text{SL}(r, {\mathbb C}))^{ir}/\text{SL}(r, {\mathbb C})$ is a 
smooth affine variety defined over the complex numbers.

We construct a holomorphic $2$--form on ${\mathcal R}$ (see Theorem \ref{thm1}), which is denoted by
$\mathcal B$. We prove that the $2$--form $\mathcal B$ on $\mathcal R$ is closed (see Theorem \ref{thm2}).
Finally we show that the restriction of $\mathcal B$ to
$$
\text{Hom}(\pi_1(M,\, x_0), \, \text{SU}(r)^{ir}\ \subset\ \mathcal R
$$
is a symplectic form (see Theorem \ref{thm3}).

This work has been much inspired by \cite{Go} and \cite{AB}. Our construction of the form $\mathcal B$
is a three-dimensional generalization of the construction of Goldman on the character variety of
a compact oriented surface (see \cite{Go}). The proof of Theorem \ref{thm3} uses the ideas of \cite{AB}.

The symplectic structure also arises from a different setting, in the context of Yang-Mills gauge theory.

\section{Sasakian manifolds}

Let $M$ be a $(2n+1)$-dimensional real $C^\infty$ orientable manifold. Let $TM_{\mathbb 
C}\,=\, TM\otimes_{\mathbb R}{\mathbb C}$ be its complexified tangent bundle. The Lie 
bracket operation on the locally defined vector fields on $M$ extends to a Lie bracket 
operation on the locally defined $C^\infty$ sections of $TM_{\mathbb C}$. A complex
subbundle of $TM_{\mathbb C}$ whose sheaf of sections are closed under the Lie bracket
operation is called \textit{integrable}.

A {\em CR-structure} on $M$ is an 
$n$-dimensional complex sub-bundle $T^{1,0}M$ of $TM_{\mathbb C}$ such that
$T^{1,0}M\cap \overline{T^{1,0}M}
\,=\,\{0\}$ and $T^{1,0}M$ is integrable. Given such a subbundle $T^{1,0}M$, there is 
a unique sub-bundle $S$ of rank $2n$ of the real tangent bundle $TM$ together with a
vector bundle homomorphism $I\,:\,S\,\longrightarrow\, S$ satisfying the following two conditions:
\begin{enumerate}
\item $I^{2}\,=\,-{\rm Id}_{S}$, and

\item $T^{1,0}M$ is the $\sqrt{-1}$--eigenbundle of $I$ acting on $S\otimes_{\mathbb R}{\mathbb C}$.
\end{enumerate}
The subbundle $\overline{T^{1,0}M}$ of $TM_{\mathbb C}$ will be denoted by $T^{0,1}M$.

A $(2n+1)$-dimensional manifold $M$ equipped with a triple $(T^{1,0}M,\, S,\, I)$ as above is 
called a {\em CR-manifold}. A {\em contact CR-manifold} is a CR-manifold $M$ with a contact 
$1$-form $\eta$ on $M$ such that ${\rm kernel}(\eta)\,=\,S$. The Reeb vector field is the unique
vector field on $\xi$ on $M$ such that $i_\xi (d\eta)\,=\, 0$ and $i_\xi \eta\,=\, 1$.
On a contact CR-manifold $M$, the above homomorphism $I$ extends to entire 
$TM$ by setting $I(\xi)\,=\,0$.

\begin{definition}\label{def1}
A contact CR-manifold $(M,\, (T^{1,0}M,\, S,\, I),\, (\eta,\, \xi))$ is a {\em 
strongly pseudo-convex CR-manifold} if the Hermitian form $L_{\eta}$ on $S_x$ defined by 
$$L_{\eta}(X,\,Y)\,=\,d\eta(X,\, IY),\ \ \, X,\,Y\,\in\, S_{x},$$
is positive definite for every point $x\,\in\, M$. 
\end{definition}

Given any strongly pseudo-convex CR-manifold $(M, \,(T^{1,0}M,\, S, \,I),\, (\eta,\, \xi))$, there
is a canonical Riemannian metric $g_{\eta}$ on $M$ which is defined to be
$$
g_{\eta}(X,Y)\ :=\ L_{\eta}(X,Y)+\eta(X)\eta(Y),\ \ \, X,\,Y\,\in \,T_{x}M,\, \ x\, \in\, M .
$$ 

\begin{definition}\label{def2}
A {\em Sasakian manifold} is a strongly pseudo-convex CR-manifold $$(M, \,(T^{1,0},\, S,\, I),\,
(\eta,\, \xi))$$ satisfying the condition that
$$[\xi,\, C^{\infty}(M;\, T^{1,0}M)]\,\subset\, C^{\infty}(M;\, T^{1,0}M).$$
In this case, the canonical metric $g_{\eta}$ is called the Sasakian metric.
\end{definition}

For the Sasakian metric $g_{\eta}$, the Reeb vector field $\xi$ is Killing and also $\vert 
\xi\vert\,=\,1$. For a Sasakian manifold $(M, \,(T^{1,0},\, S,\, I),\, (\eta,\, \xi))$, the 
Reeb vector filed $\xi$ induces a $1$-dimensional foliation; it is called
the Reeb foliation. The Reeb foliation will be denoted by
\begin{equation}\label{rf}
{\mathcal F}_{\xi}\ \subset\ TM.
\end{equation}
The sub-bundle ${\mathcal G}_{\xi}\,=\,T^{1,0}M\oplus T{\mathcal 
F}_{\xi}\subset TM_{\mathbb C}$ produces a transversely holomorphic
structure for the foliation ${\mathcal F}_{\xi}$. The form $d\eta$ is a transversely
K\"ahler structure. See \cite{BG} and references therein for Sasakian manifolds.

A compact Sasakian manifold $$(M,\, (T^{1,0}, \,S,\, I),\, (\eta, \,\xi))$$ is called {\em 
quasi-regular} if every leaf of the foliation ${\mathcal F}_{\xi}$ is closed. For any given 
compact Sasakian manifold $(M,\, (T^{1,0}, \,S,\, I),\, (\eta, \,\xi))$, there is another 
contact form $\eta^{\prime}$ with the Reeb vector field $\xi^{\prime}$ so that $(M,\, (T^{1,0}, 
\,S,\, I),\, (\eta^{\prime}, \,\xi^{\prime}))$ is quasi-regular \cite[Section 8.2.3]{BG}; see 
also \cite{Ru}, \cite{OV}. More precisely, we can take a Killing vector field $\chi$ commuting 
with $\xi$ such that $\xi^{\prime}\,=\,\xi+\chi$ and 
$\eta^{\prime}\,=\,\frac{\eta}{1+\eta(\chi)}$.

We give some simple examples of Sasakian manifolds; many more can be found in \cite{BG}.

Take a complex projective manifold $X$. Chose a K\"ahler form $\omega$ on $X$ such that the cohomology class
of $\omega$ --- which is an element of $H^2(X,\, {\mathbb R})$ --- lies in the image of $H^2(X,\, {\mathbb Z})$.
Then there is a holomorphic line bundle $L$ on $X$ equipped with a Hermitian
structure $h$ such that the curvature of the corresponding Chern connection $\nabla^{L,h}$ is
$\omega$. Consider the unit circle bundle
\begin{equation}\label{f1}
M\ :=\ \{v\,\in\, L\ \big\vert\ \, ||v||_h\,=\, 1\}\ \subset\ L .
\end{equation}
The connection $\nabla^{L,h}$ on $L$ preserves $M$. Therefore, the tangent bundle $TM$ decomposes
into a direct sum of horizontal and vertical tangent bundles. The horizontal tangent subbundle of
$TM$ will be denoted by $S$. The multiplication action of the unit circle $S^1\,=\, {\rm U}(1)$ on $M$
produces a vertical vector field; it will be denoted by $\xi$. Using the decomposition $TM\,=\, S\oplus
{\mathbb C}\cdot\xi$, the K\"ahler form $\omega$ and the
standard metric on $S^1$ together produce a Riemannian metric on $M$; this
Riemannian metric will be denoted by $g_\eta$. More precisely, $\omega$ produces a
positive symmetric form on $S$ (as $S$ is identified with the pullback of the tangent
bundle of $X$) and standard metric on $S^1$ 
produces a positive symmetric form on $T{\mathcal F}_{\xi}\,:=\, {\mathbb C}\cdot\xi$ (using the multiplication
action of $S^1\,=\,{\rm U}(1)$ on $M$); the metric $g_\eta$ is the direct sum of these two.
The connection form defining $\nabla^{L,h}$ produces a contact form $\eta$ on $M$; so the kernel of
$\eta$ is the horizontal tangent bundle. The almost complex structure on $X$ produces a homomorphism
$I\, :\, S\, \longrightarrow\, S$ such that $I^2\,=\, -{\rm Id}_S$. These data $(M,\, g,\, S,\,\xi,\,
\eta,\, I)$ define a Sasakian manifold. Every leaf of the foliation $\mathcal F_{\xi}$ is closed; in fact,
$M$ is a principal $S^1$--bundle over the leaf-space $X$. These are called regular Sasakian manifolds.

Suppose a finite group $\Gamma$ acts on $X$ via holomorphic isometries (for the K\"ahler
structure $\omega$) such that the action of lift to an action of $\Gamma$ on $L$ via holomorphic
isometries (for $h$). So the action of $\Gamma$ on $L$ preserves $M$ defined in \eqref{f1}. Assume that the
action of $\Gamma$ on $M$ is free. Note that the action of $\Gamma$ on $X$ need not be free. Then $M/\Gamma$
is a Sasakian manifold. In fact, it is a quasi-regular Sasakian manifold. If the action of $\Gamma$ on $X$
is not free, then $M/\Gamma$ is not a regular Sasakian manifold.

Here we are interested in Sasakian manifolds of dimension three. In other words, it will
be assumed that $n\,=\,1$.

\section{Flat bundles}

Let $M$ be a compact Sasakian manifold of dimension three. Fix a base point $x_0\, \in\, M$.
Fix an integer $r\, \geq\,2$. Given a homomorphism of the fundamental group
$$
\rho\ :\ \pi_1(M,\, x_0)\ \longrightarrow\ \text{SL}(r, {\mathbb C}).
$$
the standard action of $\text{SL}(r, {\mathbb C})$ on ${\mathbb C}^r$ produces an
action of $\pi_1(M,\, x_0)$ on ${\mathbb C}^r$; more precisely, the action of any $z\,\in\,
\pi_1(M,\, x_0)$ on ${\mathbb C}^r$ coincides with the action of $\rho(z)$. The homomorphism
$\rho$ is called \textit{irreducible} (also called \textit{simple}) if no nonzero proper subspace
of ${\mathbb C}^r$ is preserved by the action of $\pi_1(M,\, x_0)$ on ${\mathbb C}^r$.
The conjugation action of $\text{SL}(r, {\mathbb C})$ on itself produces an action of
$\text{SL}(r, {\mathbb C})$ on the space of all homomorphisms $\text{Hom}(\pi_1(M,\, x_0),
\, \text{SL}(r, {\mathbb C}))$; more precisely, the action of any matrix $A\,\in\, \text{SL}(r, {\mathbb C})$
on $\text{Hom}(\pi_1(M,\, x_0), \, \text{SL}(r, {\mathbb C}))$ sends any $\rho\, \in\,
\text{Hom}(\pi_1(M,\, x_0), \, \text{SL}(r, {\mathbb C}))$ to the homomorphism
$\pi_1(M,\, x_0)\, \longrightarrow\, \text{SL}(r, {\mathbb C})$ defined by
\begin{equation}\label{c1}
\gamma\ \longmapsto\
A^{-1}\rho(\gamma)A\ \in\ \text{SL}(r, {\mathbb C}), \ \ \gamma\ \in\ \pi_1(M,\, x_0).
\end{equation}

Let $$\text{Hom}(\pi_1(M,\, x_0), \, \text{SL}(r, {\mathbb C}))^{ir}\, \subset\,
\text{Hom}(\pi_1(M,\, x_0), \, \text{SL}(r, {\mathbb C}))$$ be the space of irreducible homomorphisms.
Note that the above action of $\text{SL}(r, {\mathbb C})$ on $\text{Hom}(\pi_1(M,\, x_0), \,
\text{SL}(r, {\mathbb C}))$ preserves $\text{Hom}(\pi_1(M,\, x_0), \, \text{SL}(r, {\mathbb C}))^{ir}$.
Let
\begin{equation}\label{e2}
{\mathcal R}\ :=\ \text{Hom}(\pi_1(M,\, x_0), \, \text{SL}(r, {\mathbb C}))^{ir}/\text{SL}(r, {\mathbb C})
\end{equation}
be the quotient space for this action of $\text{SL}(r, {\mathbb C})$ on $\text{Hom}(\pi_1(M,\, x_0),
 \, \text{SL}(r, {\mathbb C}))^{ir}$. This $\mathcal R$ is a finite dimensional smooth
complex manifold \cite{Go},
\cite[Corollary 1.3]{Ka}. It should be clarified that $\mathcal R$ need not be connected.
We will describe the holomorphic tangent bundle of $\mathcal R$.

We note that $\mathcal R$ has a natural algebraic structure given by the algebraic structure of
$\text{SL}(r, {\mathbb C})$ and the fact that the fundamental group $\pi_1(M,\, x_0)$ is finitely
presented. In fact, $\mathcal R$ is a smooth affine scheme defined over $\mathbb C$, as
$\text{SL}(r, {\mathbb C})$ is an affine variety.

Take any
\begin{equation}\label{e3}
\rho\ \in\ {\mathcal R},
\end{equation}
where $\mathcal R$ is constructed in \eqref{e2}. Choose a
\begin{equation}\label{e3b}
\rho'\ \in\ \text{Hom}(\pi_1(M,\, x_0), \, \text{SL}(r, {\mathbb C}))^{ir}
\end{equation}
that projects to $\rho$ under the quotient map $\text{Hom}(\pi_1(M,\, x_0), \,
\text{SL}(r, {\mathbb C}))^{ir}\, \longrightarrow\,{\mathcal R}$. Let
$$
\varpi\ :\ \widetilde{M}\ \longrightarrow\ M
$$
be the universal cover of $M$ corresponding to the base point $x_0$. Consider the
trivial holomorphic vector bundle $\widetilde{M}\times sl(r, {\mathbb C})
\, \longrightarrow\, \widetilde{M}$, where $sl(r, {\mathbb C})$ is the Lie algebra
of $\text{SL}(r, {\mathbb C})$ consisting of trace zero $r\times r$ matrices with complex entries.
The action of the Galois group $\pi_1(M,\, x_0)$ on $\widetilde M$ and the action of $\pi_1(M,\, x_0)$
on $sl(r, {\mathbb C})$ --- given by $\rho'$ (see \eqref{e3b}) and the adjoint action of $\text{SL}(r,
{\mathbb C})$ on $sl(r, {\mathbb C})$ --- together produce an action of $\pi_1(M,\, x_0)$ on 
$\widetilde{M}\times sl(r, {\mathbb C})$. The quotient
\begin{equation}\label{e4}
{\mathcal E}^{\rho'}
\ :=\ (\widetilde{M}\times sl(r, {\mathbb C}))/\pi_1(M,\, x_0)\ \longrightarrow\
{\widetilde M}/\pi_1(M,\, x_0)\ =\ M
\end{equation}
is a holomorphic vector bundle over $M$ whose fibers are Lie algebras isomorphic to
$sl(r, {\mathbb C})$
(see \cite{BS}, \cite{BK} for holomorphic vector bundles on Sasakian manifolds).
The trivial connection on the trivial vector bundle $\widetilde{M}\times sl(r, {\mathbb C})
\, \longrightarrow\, \widetilde{M}$ is preserved by the action of $\pi_1(M,\, x_0)$. Hence this
trivial connection produces a flat connection on the vector bundle ${\mathcal E}^{\rho'}$ in \eqref{e4}.
This flat connection on ${\mathcal E}^{\rho'}$ will be denoted by $\nabla^{\rho'}$.

The above flat vector bundle $({\mathcal E}^{\rho'},\, \nabla^{\rho'})$ on $M$ does not
depend on the choice of the lift $\rho'$ in \eqref{e3b} (of $\rho$). To see this, substitute $\rho'$
by the homomorphism $\rho_1\, :\, \pi_1(M,\, x_0)\, \longrightarrow\, \text{SL}(r, {\mathbb C})$
defined by $\gamma\, \longmapsto\, A^{-1}\rho'(\gamma)A$, where $A$ is a fixed element of
$\text{SL}(r, {\mathbb C})$. Let $({\mathcal E}^{\rho_1},\, \nabla^{\rho_1})$ be the flat vector
bundle on $M$ corresponding to $\rho_1$. Consider the isomorphism
$$
\Phi\ :\ \widetilde{M}\times sl(r, {\mathbb C})
\ \longrightarrow\ \widetilde{M}\times sl(r, {\mathbb C})
$$
defined by $(z,\,w)\, \longmapsto\, (z,\, A^{-1}wA)$. It is straightforward to check that $\Phi$ intertwines
the two actions of $\pi_1(M,\, x_0)$ on $\widetilde{M}\times sl(r, {\mathbb C})$ given by $\rho'$
and $\rho_1$. Consequently, $\Phi$ produces an isomorphism of flat vector bundles
$$
\widetilde{\Phi}\ :\ ({\mathcal E}^{\rho'},\, \nabla^{\rho'})\ \longrightarrow\
({\mathcal E}^{\rho_1},\, \nabla^{\rho_1}).
$$
This prove that the flat vector bundle $({\mathcal E}^{\rho'},\, \nabla^{\rho'})$ on $M$ does not
depend on the choice of the lift $\rho'$ in \eqref{e3b} (of $\rho$).

Let
\begin{equation}\label{e5}
{\mathcal L}^\rho \ \, \longrightarrow\ \, M
\end{equation}
be the sheaf of flat sections for the connection $\nabla^{\rho'}$ on ${\mathcal E}^{\rho'}$; it is a
locally constant sheaf on $M$ whose stalks are Lie algebras isomorphic to $sl(r, {\mathbb C})$. For the point
$\rho\ \in\ {\mathcal R}$ (see \eqref{e3}), we have
\begin{equation}\label{e6}
T_\rho {\mathcal R}\ \,=\ \, H^1(M,\, {\mathcal L}^\rho)
\end{equation}
\cite{Go}, \cite{AB}. The connection operator $d^\rho_0\,:=\, \nabla^{\rho'}\,:\, C^\infty(M;\, {\mathcal
E}^{\rho'})\,\longrightarrow\, C^\infty(M;\, {\mathcal E}^{\rho'}\otimes T^*M)$ extends to an operator
\begin{equation}\label{ed}
d^\rho_i
\ :\ C^\infty(M;\, {\mathcal E}^{\rho'}\otimes \bigwedge\nolimits^i T^*M)
\ \longrightarrow \ C^\infty(M;\, {\mathcal E}^{\rho'}\otimes
\bigwedge\nolimits^{i+1} T^*M)
\end{equation}
for all $i\, \geq\, 1$. We have the complex
\begin{equation}\label{ed2}
0\, \longrightarrow\, C^\infty(M;\, {\mathcal E}^{\rho'})\,
\xrightarrow{\,\,\, d^\rho_0\,\,\,}\, C^\infty(M;\, {\mathcal E}^{\rho'}\otimes T^*M)
\end{equation}
$$
\xrightarrow{\,\,\, d^\rho_1\,\,\,}\,
C^\infty(M;\, {\mathcal E}^{\rho'}\otimes\bigwedge\nolimits^2 T^*M)
\, \xrightarrow{\,\,\, d^\rho_2\,\,\,} \, C^\infty(M;\, {\mathcal E}^{\rho'}
\otimes\bigwedge\nolimits^3 T^*M)\,\longrightarrow\, 0,
$$
where $d^\rho_0,\, d^\rho_1,\, d^\rho_2$ are defined in \eqref{ed}.
The cohomology $H^1(M,\, {\mathcal L}^\rho)$ in \eqref{e6} has the following description
\begin{equation}\label{e7}
H^1(M,\, {\mathcal L}^\rho)\ \, =\ \, \frac{\text{kernel}(d^\rho_1)}{\text{image}(d^\rho_0)}.
\end{equation}
Combining \eqref{e6} and \eqref{e7}, we have the following
\begin{equation}\label{e7b}
T_\rho {\mathcal R}\ \,=\ \, \frac{\text{kernel}(d^\rho_1)}{\text{image}(d^\rho_0)}.
\end{equation}
Using the properties of Sasakian manifolds, the expression for the tangent bundle
of $\mathcal R$ in \eqref{e7b} can be considerably simplified.

For every $i\, \geq\, 0$, let
\begin{equation}\label{e8}
C^\infty(M;\, {\mathcal E}^{\rho'}\otimes \bigwedge\nolimits^i T^*M)_\xi\ \subset\
C^\infty(M;\, {\mathcal E}^{\rho'}\otimes \bigwedge\nolimits^i T^*M)
\end{equation}
be the subspace consisting of all $s\, \in\, C^\infty(M;\,
{\mathcal E}^{\rho'}\otimes \bigwedge\nolimits^i T^*M)$ such that
$$i_\xi s\ =\ 0\ =\ i_\xi d^\rho_i(s),
$$
where $d^\rho_i$ is defined in \eqref{ed} and $i_\xi$ is the contraction of differential forms by the
vector field $\xi$ (see Definition \ref{def2}); we note that the elements of
$C^\infty(M;\, {\mathcal E}^{\rho'}\otimes \bigwedge\nolimits^i T^*M)_\xi$ are called
basic forms. Since $d^\rho_{i+1}(d^\rho_i(s))\,=\, 0$, it follows that
$$
d^\rho_i(s) \ \in\ C^\infty(M;\, {\mathcal E}^{\rho'}\otimes \bigwedge\nolimits^{i+1} T^*M)_\xi
$$
for all $s\,\in\, C^\infty(M;\, {\mathcal E}^{\rho'}\otimes \bigwedge\nolimits^i T^*M)_\xi$. Let
\begin{equation}\label{ew}
\widehat{d}^\rho_i\ :\ C^\infty(M;\, {\mathcal E}^{\rho'}\otimes \bigwedge\nolimits^i T^*M)_\xi
\ \longrightarrow\ C^\infty(M;\, {\mathcal E}^{\rho'}\otimes \bigwedge\nolimits^{i+1} T^*M)_\xi
\end{equation}
be the restriction of $d^\rho_i$ to $C^\infty(M;\, {\mathcal E}^{\rho'}\otimes \bigwedge\nolimits^i
T^*M)_\xi$. Hence the complex in \eqref{ed2} has the subcomplex
\begin{equation}\label{e9}
0\, \longrightarrow\, C^\infty(M;\, {\mathcal E}^{\rho'})_\xi\,
\xrightarrow{\,\,\, \widehat{d}^\rho_0\,\,\,}\, C^\infty(M;\, {\mathcal E}^{\rho'}\otimes T^*M)_\xi
\end{equation}
$$
\xrightarrow{\,\,\, \widehat{d}^\rho_1\,\,\,}\,
C^\infty(M;\, {\mathcal E}^{\rho'}\otimes\bigwedge\nolimits^2 T^*M)_\xi
\, \xrightarrow{\,\,\, \widehat{d}^\rho_2\,\,\,} \, C^\infty(M;\, {\mathcal E}^{\rho'}
\otimes\bigwedge\nolimits^3 T^*M)_\xi\,\longrightarrow\, 0,
$$
where $\widehat{d}^\rho_i$ are as in \eqref{ew}. From \cite[Theorem 1.1]{Ka} and \eqref{e7b}
it can be deduced that
\begin{equation}\label{e10}
T_\rho {\mathcal R}\ =\ \frac{\text{kernel}(d^\rho_1)}{\text{image}(d^\rho_0)}\ =\
\frac{\text{kernel}(\widehat{d}^\rho_1)}{\text{image}(\widehat{d}^\rho_0)},
\end{equation}
where $\widehat{d}^\rho_0,\, \widehat{d}^\rho_1$ are as in \eqref{e9}.
To see \eqref{e10}, consider the partial $\text{SL}(r,{\mathbb C})$--connection $(E,\, D_\xi)$
given by the flat vector bundle $({\mathcal E}^{\rho'},\, \nabla^{\rho'})$ on $M$. So
we have
$$
(E,\, D_\xi) \ \in\ {\mathcal M}^s_{Bflat}(E,\, D_\xi)\ \subset\ {\mathcal R}
$$
(see \cite[Theorem 1.1]{Ka} for the notation ${\mathcal M}^s_{Bflat}(E,\, D_\xi)$). Next note that
$$
T_\rho {\mathcal M}^s_{Bflat}(E,\, D_\xi) \ =\ 
\frac{\text{kernel}(\widehat{d}^\rho_1)}{\text{image}(\widehat{d}^\rho_0)},
$$
where $\widehat{d}^\rho_0,\, \widehat{d}^\rho_1$ are as in \eqref{e9}. From \cite[Theorem 1.1]{Ka} we know
that ${\mathcal M}^s_{Bflat}(E,\, D_\xi)$ is an open subset of $\mathcal R$; when $M$ is
a quasi-regular Sasakian manifold, \cite[Theorem 1.1]{Ka} was proved earlier in \cite{BM}
(see \cite[p.~3494, Proposition 3.1]{BM}). In particular, we have
\begin{equation}\label{e11}
T_\rho {\mathcal M}^s_{Bflat}(E,\, D_\xi) \ =\ T_\rho {\mathcal R}.
\end{equation}
Now \eqref{e10} follows immediately from \eqref{e11} and \eqref{e7b}.

Therefore, we have proved the following:

\begin{proposition}\label{prop1}
For any $\rho\, \in\, {\mathcal R}$, the tangent space $T_\rho {\mathcal R}$ is canonically identified
with the quotient space
$$
\frac{{\rm kernel}(\widehat{d}^\rho_1)}{{\rm image}(\widehat{d}^\rho_0)}.
$$
\end{proposition}

\begin{remark}\label{re-r}
Although it was already mentioned above, a crucial property of ${\mathcal R}$ needs to be emphasized.
We first recall a result from \cite{BM}. Let $M$ be a compact connected quasi-regular Sasakian
manifold. We recall that this implies that every leaf of the Reeb foliation ${\mathcal F}_{\xi}$ is closed
(see \eqref{rf} for the Reeb foliation). Let $\mathbb F$ denote a leaf of Reeb foliation ${\mathcal F}_{\xi}$,
so $\mathbb F$ is diffeomorphic to the circle $S^1$. Let $G$ be a connected complex reductive group, and let
$$
\rho\ :\ \pi_1(M,\, x_0)\ \longrightarrow\ G
$$
be a homomorphism such that the image $\rho(\pi_1(M,\, x_0))\, \subset\, G$ is a Zariski dense subgroup
of $G$. Note that $\rho(\pi_1({\mathbb F}))$ gives a conjugacy class of cyclic subgroups of $G$, because
$\mathbb F$ is diffeomorphic to the circle $S^1$. Then the subgroup $\rho(\pi_1({\mathbb F}))$ is
of finite order (see \cite[p.~3494, Proposition 3.1]{BM}). This implies that the
conjugacy class of $\rho(\pi_1({\mathbb F}))$ remains rigid as $\rho$ moves over a family of
representations of $\pi_1(M,\, x_0)$ in $G$. It may be mentioned that \cite[Proposition 3.1]{BM} played
a key role in entire \cite{BM}.

The above result of \cite{BM} has been generalized in \cite{Ka} for arbitrary compact connected Sasakian
manifolds; see \cite[Theorem 1.1]{Ka}. In particular, \cite[Theorem 1.1]{Ka} says the following.
Let $\mathbb F$ be a closed leaf of the Reeb foliation ${\mathcal F}_{\xi}$ of
a compact connected Sasakian manifold $M$; note that we are not assuming that $M$ is
quasi-regular, so not every leaf of the Reeb foliation ${\mathcal F}_{\xi}$ is required to be closed.
Then in any connected family of reductive representations $\rho$ of $\pi(M,\, x_0)$ in $\text{SL}(r, {\mathbb C})$,
the conjugacy class of the subgroup $\rho(\pi_1({\mathbb F})$ of $\text{SL}(r, {\mathbb C})$ remains
rigid as $\rho$ moves over the connected family of reductive representations.
\end{remark}
 
\section{Construction of a two-form}

We will construct a holomorphic $2$--form on the character variety $\mathcal R$ in
\eqref{e2}. This will be done using Proposition \ref{prop1}.

As before, take $\rho\, \in\, {\mathcal R}$. Consider the trace map
$$
sl(r, {\mathbb C})\otimes sl(r, {\mathbb C}) \ \longrightarrow\ {\mathbb C},\ \ \,
A_1\otimes A_2\, \longmapsto\, \text{trace}(A_1A_2).
$$
Since this map is preserved by the adjoint action of ${\rm SL}(r, {\mathbb C})$ on its
Lie algebra $sl(r, {\mathbb C})$, it produces a bilinear map
$$
{\rm trace}\ :\ C^\infty(M;\, {\mathcal E}^{\rho'}\otimes \bigwedge\nolimits^i T^*M)\otimes
C^\infty(M;\, {\mathcal E}^{\rho'}\otimes \bigwedge\nolimits^j T^*M)\ \longrightarrow\
C^\infty(M;\, \bigwedge\nolimits^{i+j} T^*M),
$$
where ${\mathcal E}^{\rho'}$ is the vector bundle in \eqref{e4}. Take
\begin{equation}\label{e12}
(\beta,\, \gamma)\ \in\ 
C^\infty(M;\, {\mathcal E}^{\rho'}\otimes T^*M)\oplus
C^\infty(M;\, {\mathcal E}^{\rho'}\otimes T^*M).
\end{equation}
Define
\begin{equation}\label{e13}
B(\beta,\,\gamma) \ :=\ \int_M \text{trace}(\beta\bigwedge\gamma)\bigwedge\eta\ \in \ {\mathbb C},
\end{equation}
where $\eta$ is the $1$--form in Definition \ref{def1}; note that $\text{trace}(\beta\bigwedge\gamma)
\bigwedge\eta$ is a $3$--form on the compact oriented $3$--manifold $M$.

\begin{proposition}\label{prop2}
\mbox{}
\begin{enumerate}
\item If $\beta\, \in\, {\rm image}(\widehat{d}^\rho_0)$ and $\gamma\, \in\, {\rm kernel}
(\widehat{d}^\rho_1)$ (see \eqref{ew}), then 
$$
B(\beta,\,\gamma) \ =\ 0
$$
(see \eqref{e13}).

\item If $\beta\, \in\, {\rm kernel}(\widehat{d}^\rho_1)$ and $\gamma\, \in\, 
{\rm image}(\widehat{d}^\rho_0)$, then 
$$
B(\beta,\,\gamma) \ =\ 0.
$$
\end{enumerate}
\end{proposition}

\begin{proof}
For the vector field $\xi$ in Definition \ref{def2}, we have
\begin{equation}\label{e14}
i_\xi d\eta\ =\ 0.
\end{equation}
To see this, note that $L_\xi \eta \,=\, 0$. Hence
$$
0 \,=\, i_\xi d\eta + d i_\xi \eta \,=\ i_\xi d\eta + d(1) \,=\, i_\xi d\eta.
$$

Assume that $\beta\,=\, \widehat{d}^\rho_0 (s)$ for some $s\, \in\,
C^\infty(M;\, {\mathcal E}^{\rho'})_\xi$. Also, assume that $\gamma\, \in\, {\rm kernel}
(\widehat{d}^\rho_1)$. Now,
$$
B(\beta,\,\gamma) \ :=\ \int_M \text{trace}(\beta\bigwedge\gamma)\bigwedge\eta
\ =\ \int_M \text{trace}(\widehat{d}^\rho_0 (s)\bigwedge\gamma)\bigwedge\eta
$$
$$
=\ \int_M \text{trace}(s\bigwedge d^\rho_2 (\gamma\bigwedge\eta))
$$
($d^\rho_2$ is the operator in \eqref{ed}), because $\text{trace}(d^\rho_2(s\bigwedge \gamma\bigwedge\eta))$
is an exact $3$--form on $M$, and hence
$$
\int_M \text{trace}(d^\rho_2(s\bigwedge \gamma\bigwedge\eta)) \ = \ 0.
$$
Consequently, we have
\begin{equation}\label{e15}
B(\beta,\,\gamma) \ =\ \int_M \text{trace}(s\bigwedge d^\rho_1 (\gamma)\bigwedge\eta)
- \int_M \text{trace}(s\bigwedge \gamma\bigwedge d(\eta))
\end{equation}
$$
=\ 
- \int_M \text{trace}(s\bigwedge \gamma\bigwedge d(\eta)),
$$
because $\gamma\, \in\, {\rm kernel}(\widehat{d}^\rho_1)$.

We have $i_\xi \gamma\,=\, 0$ because $\gamma\, \in\,
C^\infty(M;\, {\mathcal E}^{\rho'}\otimes T^*M)_\xi$. Combining this with \eqref{e14} we conclude that
$$
i_\xi (\gamma\bigwedge d(\eta)) \ = \ 0.
$$
This implies that
$$
i_\xi (\text{trace}(s\bigwedge \gamma\bigwedge d(\eta))) \ = \ 0.
$$
Hence we have $\text{trace}(s\bigwedge \gamma\bigwedge d(\eta))\,=\, 0$ because $\xi$ is a nowhere
vanishing vector field. In particular,
$$
\int_M \text{trace}(s\bigwedge \gamma\bigwedge d(\eta)) \ = 0.
$$
Now from \eqref{e15} it follows that $B(\beta,\,\gamma) \ =\ 0$. This proves the first statement of
the proposition.

To prove the second statement of the proposition, note that
$$
\int_M \text{trace}(\beta\bigwedge\gamma)\bigwedge\eta \ =\ -
\int_M \text{trace}(\gamma\bigwedge\beta)\bigwedge\eta .
$$
Hence the second statement follows from the first statement.
\end{proof}

\begin{theorem}\label{thm1}
The construction in \eqref{e13} produces a holomorphic $2$--form
$$
{\mathcal B}\ \in\ H^0({\mathcal R},\, \Omega^2_{\mathcal R})
$$
on the character variety $\mathcal R$ in \eqref{e2}.
\end{theorem}

\begin{proof}
{}From Proposition \ref{prop2} it follows immediately that $B$ in \eqref{e13} produces a skew-symmetric
bilinear form
$$
\bigwedge\nolimits^2 \left(\frac{{\rm kernel}(\widehat{d}^\rho_1)}{{\rm image}(\widehat{d}^\rho_0)}\right)
\ \longrightarrow\ {\mathbb C}.
$$
Using the isomorphism in Proposition \ref{prop1} this produces a skew-symmetric
bilinear form
$$
\bigwedge\nolimits^2 T_\rho {\mathcal R}\ \longrightarrow\ {\mathbb C}.
$$
The resulting $2$--form on $\mathcal R$ is evidently holomorphic (as all the constructions are
in the holomorphic category).
\end{proof}

\section{Closedness of two-form}

We will prove that the $2$--form ${\mathcal B}$ in Theorem \ref{thm1} is closed. This will be 
done using a method of Atiyah and Bott developed in \cite{AB}.

Fix a $C^\infty$ complex vector bundle $E$ of rank $r$ on $M$ such that
the line bundle $\bigwedge^r E$ is trivializable.
Fix a $C^\infty$ trivialization of the complex line bundle $\bigwedge^r E$. Let
\begin{equation}\label{a2}
D_\xi
\end{equation}
be a 
partial connection on $E$ along the Reeb foliation ${\mathcal F}_{\xi}\,=\, {\mathbb C}\cdot\xi$
such that the partial connection on $\bigwedge^r E$ induced by $D_\xi$ has the following property:
The above chosen trivialization of $\bigwedge^r E$ is flat with respect to this partial
connection on $\bigwedge^r E$.

We assume that the partial connection $D_\xi$ on $E$ in \eqref{a2}
extends to a flat connection on $E$ over entire $M$.

Let
\begin{equation}\label{e16b}
\text{ad}(E)\ \subset\ \text{End}(E)\ =\ E\otimes E^*
\end{equation}
be the subbundle of co-rank one given by the endomorphisms of trace zero. The partial connection
$D_\xi$ on $E$ along ${\mathcal F}_\xi$ (see \eqref{a2}) induces a partial connection on
the vector bundle $\text{End}(E)$. This
induced partial connection on $\text{End}(E)$ preserves the subbundle $\text{ad}(E)$ in \eqref{e16b}.
Consequently, $\text{ad}(E)$ is equipped with a partial connection along ${\mathcal F}_\xi$.

Note that the normal bundle $TM/T{\mathcal F}_\xi$ of the foliation ${\mathcal F}_\xi$ is identified
with the subbundle $S\, \subset\, TM$ (see Definition \ref{def2}). In other words, the following
composition of homomorphisms is an isomorphism:
\begin{equation}\label{a1}
S \ \hookrightarrow\ TM \ \longrightarrow\ TM/T{\mathcal F}_\xi ,
\end{equation}
where $TM\, \longrightarrow\, TM/T{\mathcal F}_\xi$ is the quotient map. The normal bundle 
$TM/T{\mathcal F}_\xi$ is equipped with a natural flat partial connection along ${\mathcal F}_\xi$ 
(it is known as the Bott partial connection); this Bott partial connection is constructed as follows.
Let $\widetilde{TM}$ (respectively, $\widetilde{T\mathcal F}_\xi$) denote the sheaf of $C^\infty$
sections of $TM$ (respectively, $T{\mathcal F}_\xi$). Consider the Lie bracket operation on the
sheaf of $C^\infty$ vector fields on $M$, and let
$$
\widetilde{T\mathcal F}_\xi \otimes_{\mathbb R} \widetilde{TM}\ \longrightarrow\ \widetilde{TM}
$$
be the restriction of this operation to $\widetilde{T\mathcal F}_\xi \otimes_{\mathbb R} \widetilde{TM}$.
It induces a homomorphism
$$
\zeta\ :\ \widetilde{T\mathcal F}_\xi \otimes_{\mathbb R} (\widetilde{TM}/\widetilde{T\mathcal F}_\xi)
\ \longrightarrow\ \widetilde{TM}/\widetilde{T\mathcal F}_\xi
$$
using the natural quotient map $\widetilde{TM}\, \longrightarrow\, 
\widetilde{TM}/\widetilde{T\mathcal F}_\xi$ and the fact that $\widetilde{T\mathcal F}_\xi$ is closed
under the Lie bracket operation of vector fields. Note that we have $\zeta((f\cdot v)\otimes w)\,=\,
f\zeta(v\otimes w)$ for all locally defined $c^\infty$ function $f$ on $M$ and all
locally defined smooth sections $v$ (respectively, $w$) of $T{\mathcal F}_\xi$ (respectively,
$(TM)/T{\mathcal F}_\xi$). Consequently, $\zeta$ produces a homomorphism
$$
\zeta'\ :\ \widetilde{TM}/\widetilde{T\mathcal F}_\xi\ \longrightarrow\
(\widetilde{TM}/\widetilde{T\mathcal F}_\xi)\otimes T{\mathcal F}^*_\xi.
$$
This $\zeta'$ is a flat partial connection on $TM/T{\mathcal F}_\xi$ along ${\mathcal F}_\xi$,
which is the Bott partial connection mentioned above.

Using the isomorphism of $TM/T{\mathcal F}_\xi$ with $S$ in
\eqref{a1}, this Bott partial connection on $TM/T{\mathcal F}_\xi$
produces a partial connection on $S$ along ${\mathcal 
F}_\xi$. This partial connection on $S$ induces a partial connection, along ${\mathcal 
F}_\xi$, on $\bigwedge^j S^*\, =\,\bigwedge^j (TM/T{\mathcal F}_\xi)^*$ for all $j\,\geq\, 1$.

Let
\begin{equation}\label{e16}
\mathcal C
\end{equation}
denote that space of all $C^\infty$ connections $\nabla$ on $E$ satisfying the following
four conditions:
\begin{enumerate}
\item The connection on $\bigwedge^r E$, induced by the connection $\nabla$ on $E$, preserves
the above chosen trivialization of $\bigwedge^r E$.

\item The restriction of $\nabla$ to the Reeb foliation ${\mathcal F}_{\xi}$ coincides with the
given partial connection $D_\xi$ in \eqref{a2}.

\item The curvature $\nabla^2 \, \in\, C^\infty(M;\, \text{ad}(E)\otimes \bigwedge\nolimits^2 T^*M)$
of $\nabla$ (see \eqref{e16b} for $\text{ad}(E)$) has the property that 
$$
i_\xi \nabla^2 \ =\ 0.
$$
Note that this condition implies that
\begin{equation}\label{e17}
\nabla^2 \, \in\, 
C^\infty(M;\, \text{ad}(E)\otimes \bigwedge\nolimits^2 (TM/T{\mathcal F}_\xi)^*)
\,=\, C^\infty(M;\, \text{ad}(E)\otimes \bigwedge\nolimits^2 S^*)
\end{equation}
(see \eqref{a1} for $TM/T{\mathcal F}_\xi$ and $S$).

\item The curvature section $\nabla^2$ in \eqref{e17} is flat with respect to the partial connection
along the foliation ${\mathcal F}_\xi$. Recall that both $\text{ad}(E)$ and $\bigwedge^2
(TM/T{\mathcal F}_\xi)^*$ are equipped with partial connections along ${\mathcal F}_\xi$, and hence
$\text{ad}(E)\otimes\bigwedge^2(TM/T{\mathcal F}_\xi)^*$ is equipped with a partial connection
along ${\mathcal F}_\xi$. The condition says that the section $\nabla^2$ is flat with
respect to this partial connection on $\text{ad}(E)\otimes\bigwedge^2(TM/T{\mathcal F}_\xi)^*$.
\end{enumerate}

The space $\mathcal C$ in \eqref{e16} is an affine space for a vector space. This will be explained below.

Consider the complex vector space $C^\infty(M;\, \text{ad}(E)\otimes T^*M)$. Let
\begin{equation}\label{e18}
{\mathcal V}\ \subset\ C^\infty(M;\, \text{ad}(E)\otimes T^*M)
\end{equation}
be the subspace consisting of all $w\, \in\, C^\infty(M;\, \text{ad}(E)\otimes T^*M)$
satisfying the following two conditions:
\begin{enumerate}
\item $i_\xi w\ =\ 0$. This condition implies that
\begin{equation}\label{e19}
w\ \in\ C^\infty(M;\, \text{ad}(E)\otimes (TM/T{\mathcal F}_\xi)^*).
\end{equation}

\item The section $w$ in \eqref{e19} is flat with respect to the partial connection along the
foliation ${\mathcal F}_\xi$. Recall that both $\text{ad}(E)$ and $(TM/T{\mathcal F}_\xi)^*$
are equipped with partial connections along ${\mathcal F}_\xi$, and hence
$\text{ad}(E)\otimes (TM/T{\mathcal F}_\xi)^*$ is equipped with a partial connection
along ${\mathcal F}_\xi$. The condition says that the section $w$ in \eqref{e19} is flat with
respect to this partial connection on $\text{ad}(E)\otimes (TM/T{\mathcal F}_\xi)^*$.
\end{enumerate}

It is now straightforward to check that $\mathcal C$ in \eqref{e16} is an affine space for
the vector space $\mathcal V$ constructed in \eqref{e18}.

\begin{remark}\label{ff}
It may be mentioned that in special cases it may happen that ${\mathcal V}\ =\ 0$. In that case,
$\mathcal C$ is a singleton set, if it is nonempty.
\end{remark}

The vector space $C^\infty(M;\, \text{ad}(E)\otimes T^*M)$ has following constant $2$--form:
Take $$\beta,\, \gamma\ \in\ C^\infty(M;\, \text{ad}(E)\otimes T^*M).$$
The $2$--form on $C^\infty(M;\, \text{ad}(E)\otimes T^*M)$ sends $\beta\bigwedge\gamma$ to
$$
\int_M \text{trace}(\beta\bigwedge\gamma)\bigwedge\eta,
$$
where $\eta$ is as in Definition \ref{def1}; note
that $\text{trace}(\beta\bigwedge\gamma)\bigwedge\eta$ is a $3$--form on $M$. Let
\begin{equation}\label{e20}
\Theta'\ \in\ \bigwedge\nolimits^2 {\mathcal V}^*
\end{equation}
be the restriction to the subspace $\mathcal V$ (see \eqref{e18}) of this $2$--form on
$C^\infty(M;\, \text{ad}(E)\otimes T^*M)$. Note that the $2$--form
on $C^\infty(M;\, \text{ad}(E)\otimes T^*M)$ is closed because it is a constant
form; hence the form $\Theta'$ on $\mathcal V$ is also closed. Since $\mathcal C$ in \eqref{e16}
is an affine space for
the vector space $\mathcal V$, the closed $2$--form $\Theta'$ in \eqref{e20} produces a closed
$2$--form on $\mathcal C$. Let
\begin{equation}\label{e21}
\Theta\ \, \in\ \, H^0({\mathcal C},\, \Omega^2_{\mathcal C})
\end{equation}
be the closed $2$--form on $\mathcal C$ given by $\Theta'$.

Let
\begin{equation}\label{e22}
{\mathcal C}_0\ \,\subset\ \, {\mathcal C}
\end{equation}
be the locus of irreducible flat connections. In other words, a connection $\nabla\, \in\, {\mathcal C}$ lies
in ${\mathcal C}_0$ if and only if the following two conditions hold:
\begin{enumerate}
\item The curvature of $\nabla$ vanishes identically, and

\item the flat connection $\nabla$ is irreducible, which means that no nonzero proper subbundle of
$E$ is preserved by the connection $\nabla$.
\end{enumerate}
Let
\begin{equation}\label{e23}
\Theta_0\ \, \in\ \, H^0({\mathcal C}_0,\, \Omega^2_{{\mathcal C}_0})
\end{equation}
be the restriction of the $2$--form $\Theta$ (see \eqref{e21}) to the subspace ${\mathcal C}_0$ in \eqref{e22}.
We note that the $2$--form $\Theta_0$ is closed because $\Theta$ is so.

Recall $\mathcal R$ constructed in \eqref{e2}. Consider the map
\begin{equation}\label{e24}
\Phi\ :\ {\mathcal C}_0\ \longrightarrow\ {\mathcal R}
\end{equation}
that sends an irreducible flat connection to its monodromy homomorphism. The map $\Phi$ is a surjective
submersion to a connected component of $\mathcal R$. Let
$$
{\mathcal R}'\ :=\ \Phi ({\mathcal C}_0) \ \subset\ {\mathcal R}
$$
be the connected component of $\mathcal R$ given by the image of the map $\Phi$ in \eqref{e24}.
Next note that
\begin{equation}\label{e25}
\Phi^*{\mathcal B}\ =\ \Theta_0,
\end{equation}
where $\mathcal B$ and $\Theta_0$ are the $2$--forms in Theorem \ref{thm1} and \eqref{e23} respectively.
Since $\Phi$ in \eqref{e24} is a surjective submersion to ${\mathcal R}'$, and the $2$--form $\Theta_0$
is closed, from \eqref{e25} it follows immediately that the $2$--form ${\mathcal B}\big\vert_{{\mathcal R}'}$
on ${\mathcal R}'$ is closed.

Replacing the pair $(E,\, D_\xi)$ (see \eqref{a2}) with other pairs satisfying the same conditions, we can
cover all connected components of $\mathcal R$. Therefore, we have proved the following:

\begin{theorem}\label{thm2}
The $2$--form $\mathcal B$ in Theorem \ref{thm1} is closed.
\end{theorem}

\section{Nondegeneracy of two-form}

Consider the special unitary group $\text{SU}(r)\, \subset\, \text{SL}(r,{\mathbb C})$. As before,
given a homomorphism
$$
\rho\ :\ \pi_1(M,\, x_0)\ \longrightarrow\ \text{SU}(r),
$$
the standard action of $\text{SU}(r)$ on ${\mathbb C}^r$ produces an
action of $\pi_1(M,\, x_0)$ on ${\mathbb C}^r$; in other words, the action of any $z\,\in\,
\pi_1(M,\, x_0)$ on ${\mathbb C}^r$ coincides with the action of $\rho(z)$. The homomorphism
$\rho$ is called \textit{irreducible} (also called \textit{simple}) if no nonzero proper subspace
of ${\mathbb C}^r$ is preserved by the action of $\pi_1(M,\, x_0)$ on ${\mathbb C}^r$.

Let $$\text{Hom}(\pi_1(M,\, x_0), \, \text{SU}(r))^{ir}\, \subset\,
\text{Hom}(\pi_1(M,\, x_0), \, \text{SU}(r))$$ be the space of all irreducible homomorphisms from
$\pi_1(M,\, x_0)$ to $\text{SU}(r)$. The conjugation action of $\text{SU}(r)$ on
$\text{Hom}(\pi_1(M,\, x_0), \,
\text{SU}(r))$ (described in \eqref{c1}) preserves the space of irreducible
homomorphisms $\text{Hom}(\pi_1(M,\, x_0), \, \text{SU}(r))^{ir}$. Let
\begin{equation}\label{e26}
{\mathcal R}_U\ :=\ \text{Hom}(\pi_1(M,\, x_0), \, \text{SU}(r))/\text{SU}(r)
\end{equation}
be the quotient space for this action of $\text{SU}(r)$ on $\text{Hom}(\pi_1(M,\, x_0),
\, \text{SU}(r))^{ir}$. The inclusion map $\text{SU}(r)\, \hookrightarrow\,
\text{SL}(r, {\mathbb C})$ produces a map
\begin{equation}\label{e27}
\Psi\ :\ {\mathcal R}_U \ \longrightarrow\ {\mathcal R},
\end{equation}
where $\mathcal R$ is constructed in \eqref{e2}.

\begin{lemma}\label{lem2}
Consider the form $\mathcal B$ in Theorem \ref{thm1}.
The pullback $\Psi^*{\mathcal B}$, where $\Psi$ is the map in \eqref{e27}, is a closed $2$--form on
${\mathcal R}_U$.
\end{lemma}

\begin{proof}
The $2$--form $\mathcal B$ is closed by Theorem \ref{thm2}. Hence its pullback
$\Psi^*{\mathcal B}$ is a closed $2$--form on ${\mathcal R}_U$.
\end{proof}

Fix a $C^\infty$ complex vector bundle on $M$ of rank $r$ equipped with a Hermitian structure $h$.
Assume that the line bundle $\bigwedge^rE$ has a trivialization $$\varphi\ :\ M\times {\mathbb C}\
\longrightarrow\ \bigwedge\nolimits^r E$$ such that $h(\varphi(M\times \{1\}))\,= \,1$. Assume that
$E$ has a partial connection $D_\xi$ along the foliation ${\mathcal F}_\xi$ satisfying the
following two conditions:
\begin{enumerate}
\item The Hermitian structure $h$ on $E$ is flat with respect to the partial connection on
$E^*\otimes \overline{E}^*$ induced by $D_\xi$.

\item The section $\varphi(M\times \{1\})$ of $\bigwedge^rE$ is flat with respect to the
partial connection on $\bigwedge^rE$ induced by $D_\xi$.
\end{enumerate}

Let
\begin{equation}\label{a0}
\text{su}(E)\ \subset\ \text{ad}(E)
\end{equation}
be the real subbundle given by the skew-symmetric endomorphisms (with respect to the
Hermitian structure $h$) of trace zero. So for any $x\, \in\, M$, an endomorphism
$A\, \in\, {\rm ad}(E)_x\, \subset\, \text{End}(E_x)$ lies in $\text{su}(E)_x$ if and
only if
$$
h(A(v),\, w) + h(v,\, A(w))\ = \ 0
$$
for all $v,\, w\, \in\, E_x$. Since $h$ is flat with respect to the partial connection
induced by $D_\xi$, it follows immediately that the partial connection on $\text{ad}(E)$ induced
by $D_\xi$ preserves the subbundle $\text{su}(E)$ in \eqref{a0}.

Let
\begin{equation}\label{e28}
{\mathcal C}_U
\end{equation}
denote that space of all $C^\infty$ connections $\nabla$ on $E$ satisfying the following
five conditions:
\begin{enumerate}
\item The Hermitian structure $h$ on $E$ is flat with respect to the connection on $E^*\otimes \overline{E}^*$
induced by $\nabla$.

\item For the connection on $\bigwedge^r E$ --- induced by the connection $\nabla$ on $E$ ---
the section $\varphi(M\times \{1\})$ of $\bigwedge^rE$ is flat.

\item The restriction of $\nabla$ to the Reeb foliation ${\mathcal F}_{\xi}$ coincides with the
given partial connection $D_\xi$ in \eqref{a2}.

\item The curvature $\nabla^2 \, \in\, C^\infty(M;\, \text{su}(E)\otimes \bigwedge\nolimits^2 T^*M)$
of $\nabla$ (see \eqref{e16b} for $\text{su}(E)$) has the property that 
$$
i_\xi \nabla^2 \ =\ 0.
$$
Note that this condition implies that
\begin{equation}\label{e29}
\nabla^2 \, \in\, 
C^\infty(M;\, \text{su}(E)\otimes \bigwedge\nolimits^2 (TM/T{\mathcal F}_\xi)^*)
\,=\, C^\infty(M;\, \text{su}(E)\otimes \bigwedge\nolimits^2 S^*)
\end{equation}
(see \eqref{a1} for $TM/T{\mathcal F}_\xi$ and $S$).

\item The curvature section $\nabla^2$ in \eqref{e29} is flat with respect to the partial connection
along the foliation ${\mathcal F}_\xi$. Recall that both $\text{su}(E)$ and $\bigwedge^2
(TM/T{\mathcal F}_\xi)^*$ are equipped with partial connections along ${\mathcal F}_\xi$, and hence
$\text{su}(E)\otimes\bigwedge^2(TM/T{\mathcal F}_\xi)^*$ is equipped with a partial connection
along ${\mathcal F}_\xi$. The condition says that the section $\nabla^2$ is flat with
respect to this partial connection on $\text{su}(E)\otimes\bigwedge^2(TM/T{\mathcal F}_\xi)^*$.
\end{enumerate}

Consider the real vector space $C^\infty(M;\, \text{su}(E)\otimes T^*M)$. Let
\begin{equation}\label{e30}
{\mathcal V}_U\ \subset\ C^\infty(M;\, \text{su}(E)\otimes T^*M)
\end{equation}
be the subspace consisting of all $w\, \in\, C^\infty(M;\, \text{su}(E)\otimes T^*M)$
satisfying the following two conditions:
\begin{enumerate}
\item $i_\xi w\ =\ 0$. Note that this condition implies that
\begin{equation}\label{e31}
w\ \in\ C^\infty(M;\, \text{su}(E)\otimes (TM/T{\mathcal F}_\xi)^*).
\end{equation}

\item The section $w$ in \eqref{e31} is flat with respect to the partial connection along the
foliation ${\mathcal F}_\xi$. Recall that both $\text{su}(E)$ and $(TM/T{\mathcal F}_\xi)^*$
are equipped
with partial connections along ${\mathcal F}_\xi$, and hence
$\text{su}(E)\otimes (TM/T{\mathcal F}_\xi)^*$ is equipped with a partial connection
along ${\mathcal F}_\xi$. The condition says that the section $w$ in \eqref{e31} is flat with
respect to this partial connection on $\text{su}(E)\otimes (TM/T{\mathcal F}_\xi)^*$.
\end{enumerate}

Note that ${\mathcal C}_U$ in \eqref{e28} is an affine space for
the real vector space ${\mathcal V}_U$ constructed in \eqref{e30}.

The vector space $C^\infty(M;\, \text{su}(E)\otimes T^*M)$ has following constant $2$--form:
Take $$\beta,\, \gamma\ \in\ C^\infty(M;\, \text{su}(E)\otimes T^*M).$$
The $2$--form on $C^\infty(M;\, \text{ad}(E)\otimes T^*M)$ sends $\beta\bigwedge\gamma$ to
$$
\int_M \text{trace}(\beta\bigwedge\gamma)\bigwedge\eta ,
$$
where $\eta$ is the $1$--form in Definition \ref{def1}. Restrict
this $2$--form to the subspace ${\mathcal V}_U$ in \eqref{e30}; this restriction
will be denoted by ${\mathcal H}_0$. This $2$--form ${\mathcal H}_0$ on ${\mathcal V}_U$
is actually nondegenerate. Indeed, for any $0\, \not=\, \beta\, \in\, {\mathcal V}_U$, we have
\begin{equation}\label{nd}
\int_M \text{trace}(\beta\bigwedge\overline{\beta})\bigwedge\eta\ <\ 0.
\end{equation}
The constant $2$--form ${\mathcal H}_0$ on ${\mathcal V}_U$ produces a $2$--form on
${\mathcal C}_U$, because ${\mathcal C}_U$ is an affine space for
the real vector space ${\mathcal V}_U$. Let
\begin{equation}\label{e32}
{\mathcal H} \ \in\, C^\infty({\mathcal C}_U,\, \bigwedge\nolimits^2 T^*{\mathcal C}_U)
\end{equation}
be the $2$--form on ${\mathcal C}_U$ given by ${\mathcal H}_0$.

Let
$${\mathcal G}_U \ \subset\ C^\infty(M;\, \text{End}(E))$$
be the space of all endomorphisms $A\, \in\, C^\infty(M;\, \text{End}(E))$
satisfying the following two conditions:
\begin{enumerate}
\item For every $x\, \in\, M$, the endomorphism $A(x)\,\in\, \text{End}(E_x)$ preserves the
Hermitian structure $h$; note that this implies that $A(x)$ is actually invertible.

\item $A$ is flat with respect to the partial connection on $\text{End}(E)$ induced
by the partial connection $D_\xi$ on $E$.
\end{enumerate}
Note that ${\mathcal G}_U$ is a group. The Lie algebra of ${\mathcal G}_U$ is
the subspace
$$
C^\infty(M;\, \text{su}(E))_\xi\ \subset\ C^\infty(M;\, \text{su}(E))
$$
consisting of all sections that are flat with respect to the partial connection on
$\text{su}(E)$ induced by $D_\xi$.

The group ${\mathcal G}_U$ has a natural action
on ${\mathcal C}_U$. This action of ${\mathcal G}_U$ on ${\mathcal C}_U$ evidently
preserves the $2$--form $\mathcal H$ \eqref{e32}.

Let
$$
{\mathcal W}\ \subset\ C^\infty(M;\, \text{su}(E)\otimes \bigwedge\nolimits^2 T^*M)
$$
be the subspace consisting of all $w\, \in\, C^\infty(M;\, \text{su}(E)
\otimes \bigwedge\nolimits^2 T^*M)$ satisfying the following two conditions:
\begin{enumerate}
\item $i_\xi w\ =\ 0$. Note that this condition implies that
$$
w\ \in\ C^\infty(M;\, \text{su}(E)\otimes \bigwedge\nolimits^2 (TM/T{\mathcal F}_\xi)^*).
$$

\item The above section $w$ is flat with respect to the partial connection on
$\text{su}(E)\otimes \bigwedge\nolimits^2 (TM/T{\mathcal F}_\xi)^*$ along the
foliation ${\mathcal F}_\xi$.
\end{enumerate}

Then we have
$$
C^\infty(M;\, \text{su}(E))^*_\xi\ =\ {\mathcal W}.
$$
The duality pairing is given by
$$
-\int_M \text{trace}(\beta\delta)\bigwedge\eta,
$$
where $\beta\, \in\, C^\infty(M;\, \text{su}(E))_\xi$, $\delta\, \in\,{\mathcal W}$ and
$\eta$ is the $1$--form in Definition \ref{def1}.

It is straightforward to check that the map
$$
{\mathcal C}_U\ \longrightarrow\ \text{Lie}({\mathcal G}_U)^*\ =\
C^\infty(M;\, \text{su}(E))^*_\xi \ =\ {\mathcal W}
$$
that sends a connection $\nabla\,\in\, {\mathcal C}_U$ to its curvature
$\nabla^2\, \in\, {\mathcal W}$ is the moment map for the above action of
${\mathcal G}_U$ on ${\mathcal C}_U$.

Now from the general principle of symplectic reduction (see \cite{AB}) it
follows that the $2$--form $\Psi^*{\mathcal B}$ on ${\mathcal R}_U$ (see Lemma
\ref{lem2}) is nondegenerate. Combining this with Lemma \ref{lem2} we have the following:

\begin{theorem}\label{thm3}
The $2$--form $\Psi^*{\mathcal B}$ on ${\mathcal R}_U$ is a symplectic form.
\end{theorem}

\begin{remark}
We were compelled to restrict to the $\text{SU}(r)$ character variety (instead of
the $\text{SL}(r,{\mathbb C})$ character variety) because \eqref{nd} holds only if
$\text{SU}(r)$ is considered instead of $\text{SL}(r,{\mathbb C})$. It is conceivable
that nondegeneracy holds also for the $\text{SL}(r,{\mathbb C})$ character variety. We hope
to be able to address this question in future.
\end{remark}

\section{The two-form from Yang-Mills gauge theory on surfaces}

In this section we review, for reference, how the symplectic structure arises on the moduli space of flat 
connections over a compact oriented surface. So, instead of $M$, now we consider a compact, oriented 
2-manifold $\Sigma$ without boundary, of genus $g$ with $g\, \geq\, 2$, and consider a compact, connected Lie group $G$, whose 
Lie-algebra $LG$ is equipped with an Ad-invariant metric $\langle\cdot,\,\cdot\rangle$. For detailed results we 
refer to \cite{Sen93, Sen94, Sen95, Sen97, Sen97a, Sen97b, Se03, KS94a, KS94b, KS95, KS96}, including for 
results that apply to surfaces with boundary, non-trivial bundles, and also non-orientable surfaces.
There is a large literature in the area; we mention also \cite{Hu96} and
\cite{JeW00}, and other works by these authors, for more about the symplectic structure on the moduli space of flat connections.

Let ${\mathcal A}$ be the space of all connections on a principal $G$-bundle $\pi\,:\,P\,
\longrightarrow\, \Sigma$, and let ${\mathcal 
G}$ be the group of bundle automorphisms, meaning automorphisms of $P$ that commute with the
action of $G$ and induces the identity map $\Sigma$. There is a standard symplectic form (\cite{AB}, \cite{Go})
on the affine space 
${\mathcal A}$ given by
$$\omega_{AB}(A,B)\ :=\ \int_{\Sigma}\langle A\wedge B\rangle_{LG},$$
for all $LG$-valued $G$-equivariant $1$-forms on $P$ (the space of such $1$-forms is the tangent space to
the affine space ${\mathcal A}$). Then the action of ${\mathcal G}$ on ${\mathcal A}$ given by
$${\mathcal A}\times{\mathcal G}\, \longrightarrow\,{\mathcal A},\ \ \,
 (\omega,\, \phi)\,\longmapsto\,\phi^*\omega$$ is symplectic, with moment map given by
$J\,:\,\omega\,\longmapsto\, \Omega^\omega$, the curvature $2$-form of $\omega$ (the space of $LG$-valued
equivariant $2$-forms on $P$ can be taken to be the dual to the Lie algebra $L{\mathcal G}$).

To discuss the Yang-Mills measure, we also fix a Riemannian metric on $\Sigma$. The {\em Yang-Mills measure} $\mu_T$, 
where $T\,>\,0$, is a probability measure on a completion of ${\mathcal A}/{\mathcal G}$ given by the formal 
expression
$$d\mu_T([\omega])\ =\ \frac{1}{Z_T}e^{-\frac{S_{\rm YM}(\omega)}{T}}[D\omega],$$
where $[D\omega]$ is the ``pushforward'' of the ``Lebesgue measure'' $D\omega$ on ${\mathcal A}$ onto ${\mathcal A}/{\mathcal G}$ while $S_{\rm 
YM}(\omega)$ is the Yang-Mills action, half the $L^2$-norm-squared of the curvature form, using the Riemannian area measure on $\Sigma$, and $Z_T$ 
is a ``normalizing constant''. Note that the exponent is $-\frac{1}{2T}\|J(\omega)\|_{L^2}^2$. For convenience, assume that $G$ is simply connected, 
in which case flat connections exist on $P$. A formal calculation shows that $\mu_T\to\mu_0$, as $T\downarrow 0$, where $\mu_0$ is the normalized 
symplectic volume measure on $J^{-1}(0)/{\mathcal G}$, the moduli space of flat connections on ${\mathcal A}/{\mathcal G}$. Witten \cite{Wi91, Wi92} 
discovered this result. A mathematical proof, given in \cite{Se03} for $g\, \geq\, 1$, uses an explicit construction of the symplectic form on 
$J^{-1}(0)/{\mathcal G}$, in addition to a rigorous construction of the measure $\mu_T$ (different approaches to the construction of $\mu_T$ were 
developed by Fine \cite{Fi91, Fi96}, Sengupta \cite{Sen93, Sen97,Sen97b}, and L\'evy \cite{Le03, Le10}; ; for the volume of the moduli space of flat 
connections, see also Liu \cite{Li96} and, for a limiting result, Forman \cite{Fo93}).

Let $g$ denote the genus of $\Sigma$.
The moduli space of flat connections on $P\, \longrightarrow\, \Sigma$ can be identified concretely by
choosing standard generators $A_1,\, B_1,\,\cdots,\, A_g,\, B_g$ of $\pi_1(\Sigma,\,x_0)$ and using the
holonomy map
\[
{\mathcal A}\, \longrightarrow\, G^{2g},\ \ \,\omega\,\longmapsto\, \left(h_u(\omega; A_1), h_u(\omega; B_1),\ldots, h_u(\omega; A_g), h_u(\omega; B_g)\right).
\]
This map carries the subset of all flat connections to the set
\[ C^0\, =\, \{(a_1,\, b_1,\, \cdots,\, a_g,\, b_g)\, \in\,
G^{2g}\,\,\big\vert\,\, a_1b_1a_1^{-1}b_1^{-1}\ldots a_gb_ga_g^{-1}b_g^{-1}\,=\,e\}.\]
Passing to the quotient ${\mathcal A}/{\mathcal G}$ corresponds to the quotient $C^0/G$, where $G$ acts on
$C^0$ by conjugation in each component. For notational convenience, embed $G^{2g}$ into $G^{4g}$ by:
\[
(a_1,\,b_1,\,\cdots,\, a_g,\, b_1)\, \longmapsto\, (a_1,\, a_1^{-1},\, b_1,\, b_1^{-1},
\, \cdots,\, a_g,\, b_g,\, a_g^{-1},\, b_g^{-1}).
\]
The symplectic form then arises from the $2$-form $\Theta$ on $G^{4g}$ (pulled back to $G^{2g}$) given by
\[
\Theta_h(hX,\,hY)\,=\,\sum_{j,k=1}^{2g}\epsilon_{jk}\langle {\rm Ad}(h_{k-1}\ldots h_1)^{-1}X_k, {\rm Ad}(h_{k-1}\ldots h_1)^{-1}Y_k\rangle_{LG},
\]
where $\epsilon_{jk}\,=\,1$ if $j\,<\,k$, $\epsilon_{jk}\,=\,-1$ if $j\,>\,k$ and $\epsilon_{jk}
\,=\,0$ for $j\,=\,k$.

Now consider the case of a compact, oriented surface $\Sigma$ with boundary $\partial\Sigma$ consisting of $r$ boundary component
circles; then $\pi_1(\Sigma,\, o)$ is generated 
by loops $$A_1,\, B_1,\,\cdots,\, A_g,\, B_g,\, C_1,\, \cdots,\, C_r$$ subject to the relation
$$C_r\ldots C_1\overline{B}_g\overline{A}_gB_gA_g\ldots \overline{B}_1\overline{A}_1B_1A_1\ =\ 
I,$$
where the bar over a loop indicates the reverse loop. Let ${\Xi}\,=\,{\Xi}_1\times\ldots\times{\Xi}_r$ be a conjugacy class in $G^r$. We focus on flat connections
on $\Sigma$ with holonomy around the boundary component $C_i$ restricted to lie in the conjugacy class $C_i$; then instead of $G^{2g}$, we have
to work with ${G^{2g}\times{\Xi}}$, and the subset 
$(\Pi\big\vert_{G^{2g}\times{\Xi}})^{-1}(e)$, where $\Pi\,:\,G^{2g}\times G^r\,\longrightarrow\, G$ is constructed as follows:
\begin{equation}\label{eq:introdefPi}
\Pi(a_1,\,b_1,\,\cdots ,\, a_g,\,b_g,\, c_1,\,\cdots ,\, c_r)\ =\ c_r\ldots c_1
b_g^{-1}a_g^{-1}b_ga_g\ldots b_1^{-1}a_1^{-1}b_1a_1.
\end{equation}
Our goal is to specify a $2$-form $\Theta_{\Xi}$ on $G^{2g}\times\Xi$ whose restriction to $(\Pi\big\vert_{G^{2g}\times{\Xi}})^{-1}(e) $,
when pushed down to
the quotient $(\Pi\big\vert_{G^{2g}\times{\Xi}})^{-1}(e)/G$ is the counterpart of the Atiyah-Bott $2$-form for our surface $\Sigma$, now
with boundary. To write down an explicit formula for $\Theta_\Xi$, we need to set up some notation. We write any $\alpha\,\in\, G^{2g+r}$ as
$\alpha\,=\,(\{\alpha_j\}_{j\in J},\, \alpha_{4g+1},\, \cdots ,\ ,\alpha_{4g+r})$, where
$J$ is the indexing set
\[
J\ \stackrel{\rm def}{=} \, \{1,\,2,\, 5,\, 6,\,\cdots ,\, 4g-3,\, 4g-2\}.
\]
More explicitly,
\[
\alpha\ =\ (\alpha_1,\,\alpha_2,\,\alpha_5,\,\alpha_6,\,\cdots ,\, \alpha_{4g-3},\,\alpha_{4g-2},\, \alpha_{4g+1},\, \cdots,\, \alpha_{4g+r}).
\]
We set
\[
\alpha_3\,=\,\alpha_1^{-1},\
\alpha_4\,=\, \alpha_2^{-1},\ \cdots ,\ \alpha_{4g-1}\,=\, \alpha_{4g-3}^{-1},\ \alpha_{4g}\,=\, \alpha_{4g-2}^{-1}.
\]
We write a tangent vector $\alpha H\,\in\, T_\alpha
(G^{2g}\times G^r)$ as $(\alpha_jH_j)_{j\in J'}$, with $j$ running over
$J'\,=\, J\cup\{4g+1,\,\cdots ,\, 4g+r\}$, and we set
\[
H_{j+2}\ =\ -{\rm Ad}(\alpha_j)H_j
\]
for all $j\,\in\, J$. Let $\alpha\,\in\,\Pi^{-1}(e)\,\subset\,{G^{2g}\times\Xi}$ and $\alpha
H^{(1)},\,\alpha H^{(2)}\,\in\, T_\alpha ({G^{2g}\times\Theta})$.
Then the counterpart of the $2$-form $\Theta$ for this context is
\begin{equation}\label{eq:defOmoTh}
{ \Theta}_{\Xi}(\alpha H^{(1)},\alpha H^{(2)}) 
\ =\ \frac{1}{2}\sum_{i,j=1}^{4g+r}\epsilon_{ij}\langle
f_{i-1}^{-1}H_i^{(1)},\ f_{j-1}^{-1}
H_j^{(2)}\rangle_{LG}
\end{equation}
$$
+ \frac{1}{2}\sum_{k=1}^{r}\langle
\bigl({\rm Ad} c_k^{-1}-1\bigr)^{-1}C^{(1)}_{k}, \ \bigl({\rm Ad} c_k -{\rm Ad}
c_k^{-1}\bigr)\bigl({\rm Ad} c_k^{-1}-1\bigr)C_k^{(2)}\rangle_{LG},
$$
where we have taken $\alpha\,=\,(\{\alpha_j\}_{j\in J},\, c_1,\, \cdots ,\,c_r)$ and
set $C^{(i)}_k\,=\, H^{(i)}_{4g+k}$, and
\begin{equation}
\epsilon_{ij}\ =\ \left\{ \begin{array}{cc}
1 &\mbox{if $i<j$}\\
0 & \mbox{if $i=j$}\\
-1 & \mbox{if $i>j$}
\end{array}
\right.
\end{equation}
It is shown in \cite{Se02} that this $2$-form arises from an Atiyah-Bott-type $2$-form on the space of connections over $\Sigma$ with boundary 
holonomies restricted to lie in the conjugacy classes $\Xi_i$. Further, it is shown there that $\Theta_\Xi$ induces a symplectic form on a maximal 
stratum of the corresponding moduli space of flat connections.

\section*{Acknowledgements}

We thank the referee for helpful comments.
The first-named author acknowledges the support of a J. C. Bose Fellowship (JBR/2023/000003).

\section*{Statements and Declarations}

There is no conflict of interests regarding this manuscript. No data were generated or used.


\end{document}